\documentclass{amsart}
\usepackage{color}
\usepackage{amssymb}
\usepackage{graphicx}
 \usepackage{epstopdf}
 \usepackage{tikz-cd}
\usepackage{hyperref}
\newtheorem{theorem}{Theorem}%[section]

\theoremstyle{definition}

\newtheorem{example}[theorem]{Example}

\theoremstyle{remark}
\newtheorem{remark}{Remark}
\numberwithin{equation}{section}

\begin{document}
\title{THE NON-ITERATES ARE DENSE IN THE SPACE OF CONTINUOUS SELF-MAPS}

%    Information for first author
\author{B.V. RAJARAMA BHAT}
%    Address of record for the research reported here
\address{Indian Statistical Institute, Stat-Math. Unit, R V College Post, Bengaluru 560059, India}
%    Current address
%\curraddr{Department of Mathematics and Statistics,
%Case Western Reserve University, Cleveland, Ohio 43403}
\email{bhat@isibang.ac.in}
%    \thanks will become a 1st page footnote.
\thanks{The first author is supported by  J
	C Bose Fellowship of the Science and Engineering Board, India.}

\author{CHAITANYA GOPALAKRISHNA}
%    Address of record for the research reported here
\address{Indian Statistical Institute, Stat-Math. Unit, R V College Post, Bengaluru 560059, India}
%    Current address
%\curraddr{Department of Mathematics and Statistics,
%Case Western Reserve University, Cleveland, Ohio 43403}
\email{cberbalaje@gmail.com, chaitanya\_vs@isibang.ac.in}
%    \thanks will become a 1st page footnote.

\thanks{The second author is
	supported 
	%by the Indian Statistical Institute,  Bangalore 
	%in the form of a Visiting Scientist position 
	%through J C Bose Fellowship of the first author, and 
	by the National Board for Higher Mathematics, India through No:
	0204/3/2021/R\&D-II/7389.}

%    General info
\subjclass[2020]{Primary  39B12; Secondary 37B02.}

\date{August 08, 2022.}

%\dedicatory{This paper is dedicated to our advisors.}

\keywords{Iterative  root, piecewise affine linear, dense.}

\begin{abstract}
	In this paper we develop a tool to identify functions which have no iterative roots of any order. Using this, we prove that when $X$ is $[0,1]^m$,  $\mathbb{R}^m$ or $S^1$, every non-empty open set of the space $\mathcal{C}(X)$ of continuous self-maps on $X$ endowed with the compact-open topology contains a map that does not have even discontinuous iterative roots of order $n\ge 2$.
%Let $\mathcal{C}(X)$ denote the set of continuous self-maps on a locally compact Hausdorff space $X$ endowed with the compact-open topology. When $X$ is $[0,1]^m$,  $\mathbb{R}^m$ and $S^1$, it is proved that each non-empty open set of $\mathcal{C}(X)$ contains a map that does not have even discontinuous iterative roots of order $n\ge 2$. 
This, in particular, proves that  the complement of $\{f^n: f\in \mathcal{C}(X)~\text{and}~n\ge 2\}$, the set of non-iterates, is dense in $\mathcal{C}(X)$ for these $X$. 
\end{abstract}

\maketitle

%%%%%%%%%%%%%%%%%%%%%%%%%%%%%%%%%%%%%%%%%%%%%%%%%%%%%%%%%%%%%%%%%%%%%%%%
%%%%%%%%%%%%%%%%%%%%%%%%%%%%%%%%%%%%%%%%%%%%%%%%%%%%%%%%%%%%%%%%%%%%%%%%

\section{Introduction}

Given a self-map $f$ on a non-empty set $X$
and an integer $n\ge 2$, an {\it iterative root of order $n$} 
%(or simply {\it $n^{th}$ root}) 
of
$f$ is a map $g: X\to X$ such that
\begin{eqnarray}\label{it-root}
	g^n(x)=f(x), \quad \forall x\in X,
\end{eqnarray}
where for each non-negative integer $k$, $g^k$ denote the $k$-th order  iterate of $g$ defined by $g^0={\rm id}$,
the identity map on $X$, and $g^{k}=g\circ g^{k-1}$ for $k\ge 1$. For $n=2$,
a solution of \eqref{it-root} is called an {\it iterative square root}
%(or simply a {\it square root}) 
of $f$. Let $\mathcal{C}(X)$ denote the set of continuous self-maps on a locally compact Hausdorff space  $X$ endowed with the compact-open topology and $\mathcal{W}(X):=\cup_{n=2}^{\infty}\mathcal{W}(n;X)$, where $\mathcal{W}(n;X):=\{f^n:f\in \mathcal{C}(X)\}$ for all $n\ge 2$.  %Since iterative roots are of great interest, as they define fractional iteration, present the middle procedure for the evolution of physical processes, and proposes weaker versions of the {\it  embedding flow problem} and the {\it invariant curve problem}, it is interesting to ask how big the set $\mathcal{W}(X)$ is in $\mathcal{C}(X)$. However,  most of the results in the literature are only known for $X=[0,1]$, a compact interval in $\mathbb{R}$.
Since the iterative root problem \eqref{it-root} is of great interest, being weaker versions of the {\it embedding flow problem} (\cite{fort1955}) and the {\it invariant curve problem} (\cite{kuczma1990}), it is interesting to ask how big the set $\mathcal{W}(X)$ is in $\mathcal{C}(X)$. However,  most of the results seen in the literature are only for $X=[0,1]$, the unit interval in $\mathbb{R}$.

%For each $m\in \mathbb{N}$, consider $\mathbb{R}^m$ and the {\it unit cube} $[0,1]^m:=\{(x_1, x_2, \ldots, x_m): 0\le x_j\le 1~\text{for all}~1\le j\le m\}$ of $\mathbb{R}^m$ in the usual topology.
%Then $\mathcal{C}(\mathbb{R}^m)$, the space of continuous self-maps on the Euclidean space $\mathbb{R}^m$, is metrizable for all $m\in \mathbb{N}$ by Theorem 7 of \cite{arens1946} with metric 
%\begin{eqnarray*}\label{D}
%D(f,g):=\sum_{j=1}^{\infty}\mu_j(f,g),
%\end{eqnarray*}
%where
%\begin{eqnarray*}
%\mu_j(f,g)&:=&\min\left\{\frac{1}{2^j}, \rho_j(f,g)\right\},\quad \forall j\in \mathbb{N},\\
%\rho_j(f,g)&:=&\sup \{\|f(x)-g(x)\|_\infty: x\in K_j\},\quad \forall j\in \mathbb{N},\\
%\|(x_1,x_2,\ldots,x_m)\|_\infty&:=&\max \{|x_j|:1\le j\le m\}, \quad  \forall (x_1,x_2,\ldots,x_m)\in \mathbb{R}^m
%\end{eqnarray*}
%and $(K_j)_{j\in \mathbb{N}}$ is a sequence of compact sets in $\mathbb{R}^m$ such that
%if $K$ is any compact subset of $\mathbb{R}^m$ then $K\subseteq \cup_{j=1}^{k}K_{l_j}$ for some finitely many $K_{l_1}, K_{l_2}, \ldots, K_{l_k}$. The space $\mathcal{C}([0,1]^m)$ of all continuous self-maps on the {\it unit cube}
%\begin{eqnarray*}
%[0,1]^m:=\{(x_1, x_2, \ldots, x_m): 0\le x_j\le 1~\text{for all}~1\le j\le m\}
%\end{eqnarray*}
%% $$[0,1]^m:=\{(x_1, x_2, \ldots, x_m): 0\le x_j\le 1~\text{for all}~1\le j\le m\}$$ 
% of $\mathbb{R}^m$ is also metrizable with metric 
% \begin{eqnarray*}
	% \rho(f,g):=\sup \{\|f(x)-g(x)\|_\infty: x\in [0,1]^m\},
	% \end{eqnarray*}
%%$\rho(f,g):=\sup \{\|f(x)-g(x)\|_\infty: x\in [0,1]^m\}$, 
%which is the {\it uniform metric} on it.

Humke and Laczkovich proved in
\cite{Humke-Laczkovich1989,Humke-Laczkovich} that $\mathcal{W}(2;[0,1])$ contains no open balls of $\mathcal{C}([0,1])$ (i.e., its complement $(\mathcal{W}(2;[0,1]))^c$ is
dense in $\mathcal{C}([0,1])$),  $\mathcal{W}(2;[0,1])$ is not dense in $\mathcal{C}([0,1])$, and
$\mathcal{W}(n;[0,1])$ is an analytic non-Borel subset of $\mathcal{C}([0,1])$ for $n\ge 2$.
%\begin{description}
%   \item[(R1)] $\mathcal{W}(2)$ does not contain a ball of $\mathcal{C}([0,1])$,
%   \item[(R2)]  $\mathcal{W}(2)$ is not everywhere dense in $\mathcal{C}([0,1])$, and
%   \item[(R3)] $\mathcal{W}(n)$ is analytic non-Borel subset of $\mathcal{C}([0,1])$.
%\end{description}
Simon proved in \cite{simon1989,simon1990,simon1991a,simon1991b} that $\mathcal{W}([0,1])$ is of first category and of zero Wiener measure, $\mathcal{W}(2;[0,1])$ is nowhere
dense in $\mathcal{C}([0,1])$, and
$\mathcal{W}([0,1])$ is not  dense in $\mathcal{C}([0,1])$.
% whenever $\mathcal{C}([0,1])$ is equipped with
% the uniform metric $\rho$ defined by
%   \begin{eqnarray*}
	%       \rho(f,g)=\sup\{|f(x)-g(x)|:x\in [0,1]\}.
	%   \end{eqnarray*}
%  \begin{description}
	%   \item[(R1)] $\mathcal{W}(2)$ is nowhere dense in $\mathcal{C}([0,1])$, and
	%   \item[(R2)]  $\mathcal{W}$ is not everywhere dense in $\mathcal{C}([0,1])$.
	%  \end{description}
Later, nowhere denseness of $\mathcal{W}([0,1])$ in ${\mathcal
	C}([0,1])$ was proved by Blokh \cite{Blokh} using a method
different from that of Humke-Laczkovich and Simon.  
%\begin{theorem}{\rm (\cite{BG})} \label{SqRoot}
%	Each ball of $\mathcal{C}([0,1]^m)$ contains a map that does not have even discontinuous square roots. In particular,  $(\mathcal{W}(2;[0,1]^m))^c$  is dense in $\mathcal{C}([0,1]^m)$.
%		 Furthermore, the above results are also true when $[0,1]^m$ is replaced with $\mathbb{R}^m$.
%\end{theorem}
%\begin{center}
%that each ball of $\mathcal{C}([0,1]^m)$ (resp. $\mathcal{C}(\mathbb{R}^m)$) contains a map that does not have even discontinuous square roots, and in particular, that $(\mathcal{W}([0,1]^m))$ (resp. $(\mathcal{W}(\mathbb{R}^m))$) is dense in $\mathcal{C}([0,1]^m)$ (resp. $\mathcal{C}(\mathbb{R}^m)$) for $m\ge 1$, where $[0,1]^m$ denote 
%the {\it unit cube}
%\begin{eqnarray*}
%	I^m :=\{ (x_1, x_2, \ldots , x_m)\in {\mathbb R}^m: x_i\in I~\text{for}~ 1\leq i\leq m\}
%\end{eqnarray*}
%in $\mathbb{R}^m$. 
%\end{center}

The authors showed in \cite[Theorem 2.4]{BG} that if a self-map $f$ on a non-empty $X$ has a distinct point $x_0$ that is not a fixed point of $f$ and such that the cardinality of $f^{-2}(x_0)$ is large compared to that of $f^{-1}(x)$ for all $ x\neq x_0$, then $f$ has no iterative square roots on $X$. These results, in particular, facilitated them to construct functions in $\mathcal{C}(X)$ without iterative square roots in abundance when $X$ is $[0,1]^m$, the {\it unit cube} $\{(x_1, x_2, \ldots, x_m): 0\le x_j\le 1~\text{for}~1\le j\le m\}$ in the  Euclidean space $\mathbb{R}^m$, and $\mathbb{R}^m$ in their usual topologies (see Theorems 3.8 and 4.3 in \cite{BG}). In this paper, we 
%significantly strengthen 
prove stronger versions of
these results to show that the same approach can be used to obtain functions with no iterative roots of any order $n\ge 2$ (see Theorem \ref{Tool}). 
We also generalize the Case (i) of Theorem 2.4 in \cite{BG} to include all cases of finite cardinality (see {\bf (C1)} of Theorem \ref{Tool}).
Furthermore, as an application of Theorem \ref{Tool}, we prove the following result, which is much stronger than Theorems 3.8 and 4.3 in \cite{BG} and includes an additional case of $X=S^1$, the {\it unit circle} $\{z\in \mathbb{C}:|z|=1\}$ %$S^1:=\{e^{it}: t\in [0, 2\pi)\}$ 
in the complex plane $\mathbb{C}$ in its usual topology. 

%Consider the {\it Euclidean plane} $\mathbb{R}^m$, the {\it unit cube} $[0,1]^m:=\{(x_1, x_2, \ldots, x_m): 0\le x_j\le 1~\text{for}~1\le j\le m\}$ in $\mathbb{R}^m$ and the {\it unit circle} $S^1:=\{z\in \mathbb{C}:|z|=1\}$ %$S^1:=\{e^{it}: t\in [0, 2\pi)\}$ 
%in the complex plane $\mathbb{C}$ in their usual topologies, where $m\ge 1$.
%
%Our aim here is to prove the following result, a weaker version of which was recently proved in \cite{BG} by the authors %of this paper 
%for $[0,1]^m$ and $\mathbb{R}^m$ 
%considering only the iterative square roots. 
\begin{theorem} \label{NoRoot}
	Every non-empty open set of $\mathcal{C}(X)$ contains a map that does not have even discontinuous iterative roots of order $n\geq 2$  in the following cases: {\bf (i)} $X=[0,1]^m$ for $m\in \mathbb{N}$; {\bf (ii)} $X=\mathbb{R}^m$ for $m\in \mathbb{N}$; {\bf (iii)} $X=S^1$. In particular  $(\mathcal{W}(X))^c$ is dense in $\mathcal{C}(X)$ in these cases.
%	Each non-empty open set of $\mathcal{C}([0,1]^m)$, where $m\ge 1$, contains a map that does not have even discontinuous iterative roots of order $n\ge 2$. In particular,  $(\mathcal{W}([0,1]^m))^c$  is dense in $\mathcal{C}([0,1]^m)$.
%	Furthermore, the above  results are also true when $[0,1]^m$ is replaced by  $\mathbb{R}^m$ and $S^1$. 
\end{theorem}
%%In Section \ref{S2}, we prove a result on the non-existence of iterative roots for self-maps on arbitrary sets that is similar to but more general and stronger than Theorem 2.4 of \cite{BG}, and use it as the primary tool to prove Theorem \ref{NoRoot}.
% and serves as the primary tool in proving Theorem \ref{NoRoot} in Section \ref{S3}. %Finally, in Section \ref{S4}, we make some interesting observations about our main tool, the result proved in Section \ref{S2}. 
%Finally, in Section \ref{S4}, we make some interesting observations about Theorem \ref{Tool}, our main tool.
 In Section \ref{S4},  we provide simple examples to show the necessity of  various assumptions made in Theorem \ref{Tool}.
 One of the most important aspects of this theorem is the role of comparing the sizes of inverse images of a given function at different points in determining the absence of iterative roots. Indeed, assuming the choice of axiom, this comparison can be made using every infinite cardinal number $\aleph_\alpha$ indexed by the ordinal numbers $\alpha$ (see Theorem \ref{Tool-cardinal}).
 It is natural to ask if this comparison can be made using measure theory rather than just cardinality. Interestingly, this appears to be impossible (see Example \ref{Ex4}).

%\section{Main results}\label{sec2}
%%
%\footnote{\color{red} [C210309] Sir, I have reorganized the sections as follows: Section 2--On arbitrary sets and topological spaces, Section 3--On $I^m$, Section 4--On $\mathbb{R}^m$. For now, I have retained the discussion on $I^m$ with $\mathbb{C}(I^m)$ considered in $L_p$ norm in Section 5. I will modify later as you suggest.}
%%

%\section{$W_m(2)$ does not contain a ball}
%In this section, we study iterative roots of self-maps on arbitrary sets and continuous self-maps on general topological spaces. First, we consider $X$ to be a non-empty set with no predefined topology, and prove some new results on the existence and non-existence of iterative square roots of maps in $\mathcal{F}(X)$, the set of all self-maps of $X$.   Surprisingly, some of
%these results are useful even when $X$ has a topology and we seek for roots inside  ${\mathcal C}(X)$, the space of all continuous self-maps of $X$.

\section{A tool to detect functions without iterative roots} \label{S2}

%We first prove
%In this section we prove
%a result on the non-existence of iterative roots for self-maps on arbitrary sets that is similar to yet stronger than Theorem 2.4 of \cite{BG} and acts as the primary tool in proving our main result.  
Let $X$ be a non-empty set.  For each self-map $f$ on $X$, $A\subseteq X$ and $x\in X$,  let  $f^{-1}(x):=\{y\in X:
f(y)=x\}$, $f^{-2}(x):=\{y\in X: f^2(y)=x\}$, $R(f)$
denote the range of $f$, and
$\# A$ the cardinality of $A$. 
%In view of Theorems 3.8 and 4.3 of \cite{BG}, in order to prove Theorem \ref{NoRoot}, it suffices to prove the following result, which is indeed more stronger than Theorem 2.4 of \cite{BG}.  
% The following result, as stated in the Introduction, will be useful in proving Theorem \ref{NoRoot}, but it appears to be of interest on its own. 
The following result is not  only useful in proving Theorem \ref{NoRoot}, but also appears to be of interest on its own.

\begin{theorem}\label{Tool}
	Let $f$ be a self-map on a non-empty set $X$ such that $f(x_0)\ne x_0$ for some $x_0\in X$.
	Then $f$ has no iterative roots of order $n\ge 2$ on $X$ in the following cases:
	\begin{enumerate}
		\item[\bf (C1)] \quad $\# f^{-2}(x_0)>N^3$ and $\# f^{-1}(x)\leq N$ for all $x\neq
		x_0$ and for some $N\in \mathbb{N}$;
		
		\item[\bf (C2)] \quad $ f^{-2}(x_0)$ is infinite and $ f^{-1}(x)$ is finite for
		all $x\ne x_0$;
		
		\item[\bf (C3)] \quad  $f^{-2}(x_0)$ is uncountable and $f^{-1}(x)$ is countable for all $x\ne
		x_0$.
	\end{enumerate}
\end{theorem}
\begin{proof}
	Suppose, on the contrary,  that $f=g^n$ for some self-map $g$ on $X$ and $n\ge 2$. Consider the actions of $f$ and $g$ on various subsets of $X$ around $x_0$ given by the following diagram 
	\begin{center}
		\begin{tikzpicture}[scale=0.15]
			\tikzstyle{every node}+=[inner sep=0pt]
			%	\draw [black] (17,-3.2) circle (3);
			\draw (17,-3.2) node {$E_{-1}$};
			%	\draw [black] (17,-16.1) circle (3);
			\draw (17,-16.1) node {$E_0$};
			%	\draw [black] (31.5,-16.1) circle (3);
			\draw (31.5,-16.1) node {$A_{-1}$};
			%	\draw [black] (17,-28.4) circle (3);
			\draw (17,-28.4) node {$E_2$};
			%	\draw [black] (31.5,-28.4) circle (3);
			\draw (31.5,-28.4) node {$\{x_0\}$};
			%	\draw [black] (31.5,-3.2) circle (3);
			\draw (31.5,-3.2) node {$A_{-2}$};
			%	\draw [black] (24.5,-16.1) circle (3);
			\draw (24.5,-16.1) node {$\subseteq$};
			%	\draw [black] (45.3,-28.4) circle (3);
			\draw (45.3,-28.4) node {$\{y_0\}$};
			%	\draw [black] (3.2,-28.4) circle (3);
			\draw (3.2,-28.4) node {$E_1$};
			\draw [black] (20,-28.4) -- (28.5,-28.4);
			\fill [black] (28.5,-28.4) -- (27.7,-27.9) -- (27.7,-28.9);
			\draw (24.25,-29.8) node [below] {$g$};
			\draw [dashed] (31.5,-25.4) -- (31.5,-19.1);
			\fill [black] (31.5,-19.1) -- (31,-19.9) -- (32,-19.9);
			\draw (32,-22.25) node [right] {$f^{-1}$};
			\draw [dashed] (31.5,-13.1) -- (31.5,-6.2);
			\fill [black] (31.5,-6.2) -- (31,-7) -- (32,-7);
			\draw (32,-9.65) node [right] {$f^{-1}$};
			\draw [dashed] (17,-13.1) -- (17,-6.2);
			\fill [black] (17,-6.2) -- (16.5,-7) -- (17.5,-7);
			\draw (16.5,-9.65) node [left] {$g^{-1}$};
			\draw [black] (29.26,-5.19) -- (19.24,-14.11);
			\fill [black] (19.24,-14.11) -- (20.17,-13.95) -- (19.51,-13.2);
			\draw (23.24,-9.16) node [above] {$f$};
			\draw [black] (34.5,-28.4) -- (42.3,-28.4);
			\fill [black] (42.3,-28.4) -- (41.5,-27.9) -- (41.5,-28.9);
			\draw (38.4,-29.8) node [below] {$g$};
			\draw [black] (6.2,-28.4) -- (14,-28.4);
			\fill [black] (14,-28.4) -- (13.2,-27.9) -- (13.2,-28.9);
			\draw (10.1,-28.9) node [below] {$g^{n-2}$};
			\draw [black] (14.76,-18.1) -- (5.44,-26.4);
			\fill [black] (5.44,-26.4) -- (6.37,-26.24) -- (5.7,-25.5);
			\draw (9.31,-21.76) node [above] {$g$};
			\draw [black] (19.29,-18.04) -- (29.21,-26.46);
			\fill [black] (29.21,-26.46) -- (28.93,-25.56) -- (28.28,-26.32);
			\draw (23.24,-22.74) node [below] {$f$};
		\end{tikzpicture}
	\end{center}
	where $y_0:=g(x_0)$, $A_{-1}:=f^{-1}(x_0)$, $A_{-2}:=f^{-2}(x_0)$,
	$E_0:=f(A_{-2})$,
	$E_{-1}:=g^{-1}(E_0)$,  $E_1:=g(E_0)$ and $E_2:=g^{n-2}(E_1)$.   Since $A_{-2}\ne \emptyset$ by hypothesis, clearly $E_1\ne \emptyset$.  Also, since $f=g^n$, we have
	\begin{align}
		%	g^{n-1}(E_1)&=g(g^{n-2}(E_1))=g(E_2)=\{x_0\}, \nonumber \\
		f(E_1)&=g^2(g^{n-2}(E_1))=g^2(E_2)=g(g(E_2))=g(\{x_0\})=\{y_0\},  \label{3} \\
		f(E_{-1})&=g^{n-1}( g(E_{-1}))\subseteq g^{n-1}(E_0)=g^{n-2}(g(E_0))=g^{n-2}(E_1)=E_2. \label{2'} 
	\end{align}
	%	implying that 
	%	\begin{eqnarray}\label{2}
		%	E_{-1}\subseteq \bigcup_{x\in $E_2$}f^{-1}(x).
		%	\end{eqnarray}
	Further, since $f(x_0)\ne x_0$, we have $y_0\ne x_0$,  $x_0\notin A_{-1}$ and 
	$x_0\notin E_2$.

	\noindent {\bf (C1)} \quad Since $y_0\ne x_0$, by \eqref{3} we have $\#E_1\le N$. Therefore, as $x_0\notin E_2$ and $\#E_2\le N$, by  \eqref{2'} we have
	$\#E_{-1}\le \#(f^{-1}(E_2))\le N\cdot N=N^2$.
	
	On the other hand,
	since $x_0\notin A_{-1}$ and  $E_0\subseteq A_{-1}$, we have $\#f^{-1}(x)\le N$ for all $x\in E_0$, implying that $\#E_0>N^2$, because $\#A_{-2}>N^3$ and $A_{-2}\subseteq f^{-1}(E_0)$. 
	%(otherwise, $\#A_{-2}\le %\#(\cup_{x\in E_0}f^{-1}(x))\le
	%\sum_{j=1}^{\#E_0}N\le N^2\cdot N=N^3$, which is a contradiction). 
	% We discuss in the following two cases. 
	%\noindent Case (a): Suppose that
	Therefore, as $E_0\subseteq R(f)\subseteq R(g)$, it follows that $\#E_{-1}>N^2$.
	This contradicts an earlier conclusion.
	%	\noindent Case (b): Suppose that $\#E_1>N$. Then
	%	\begin{eqnarray}\label{T1}
		%	f(E_1)=g^n(E_1)=g^{n+1}(E_0)=g(f(E_0))=\{g(x_0)\}=\{y_0\},
		%	\end{eqnarray}
	%	% $$f(E_1)=g^n(E_1)=g^{n+1}(E_0)=g(f(E_0))=\{g(x_0)\}=\{y_0\}.$$
	%	implying that $\#f^{-1}(y_0)\ge \#E_1>N$. 
	%	Therefore, again we have a contradiction, because $y_0\ne x_0$.
	Hence $f$ has no iterative roots of order $n\ge2$ on $X$.

	\noindent {\bf (C2)} \quad  Since $y_0\ne x_0$, by \eqref{3} we have $E_1$ is finite.  Therefore, as $x_0\notin E_2$ and $E_2$ is finite, by  \eqref{2'} it follows that
	$E_{-1}$ is finite.

	On the other hand, 
	since $x_0\notin A_{-1}$ and  $E_0\subseteq A_{-1}$, we have $f^{-1}(x)$ is finite for all $x\in E_0$, implying that $E_0$ is infinite, because  $A_{-2}$ is infinite and $A_{-2}\subseteq f^{-1}(E_0)$.  Therefore, as $E_0\subseteq R(f)\subseteq R(g)$, we see that $E_{-1}$ is infinite.
	This contradicts the conclusion of the previous paragraph.
	Hence $f$ has no iterative roots of order $n\ge 2$ on  $X$.
	
	\noindent {\bf (C3)} \quad The proof of  {\bf (C2)}
	% repeatedly uses 
	is based on
	the  fact that a
	finite union of finite sets is finite. The proof of this case is similar, using the result that a countable union of
	countable sets is countable.
\end{proof}

\begin{example}
	{\rm	Consider the maps $f_1, f_2:[0,1]\to [0,1]$
		defined by 
		\begin{eqnarray*}
			f_1(x)=\left\{\begin{array}{cll}
				3x&\text{if}&0\le x\le \frac{1}{4},\\
				\frac{3}{4}&\text{if}&\frac{1}{4}\le x\le \frac{1}{2},\\
				-3x+\frac{9}{4}&\text{if}&\frac{1}{2}\le x\le \frac{3}{4},\\
				x-\frac{3}{4}	&\text{if}&\frac{3}{4}\le x\le 1
			\end{array}\right. \quad \text{and} \quad 			f_2(x)=\left\{\begin{array}{cll}
				4x&\text{if}&0\le x\le \frac{1}{4},\\
				1-x	&\text{if}&\frac{1}{4}<x\le \frac{1}{2},\\
				\frac{1}{4}&\text{if}&\frac{1}{2}< x< \frac{3}{4},\\
				2-2x	&\text{if}&\frac{3}{4}\le x\le 1
			\end{array}\right.
		\end{eqnarray*}
		(see Figure \ref{Fig1}). Then $f_1$ is continuous and do not have even discontinuous iterative roots of order $n\ge 2$ on $[0,1]$ by {\bf (C2)} of Theorem \ref{Tool}. For the same reason, the discontinuous map $f_2$ also do not have even discontinuous iterative roots of order $n\ge 2$ on $[0,1]$. }
	\begin{figure}[h]
		\begin{center}
		\begin{tikzpicture}[scale=0.9][domain=0:10]
			\draw [thick] (0,0) -- (0,5);
			\draw [thick] (0,5) -- (5,5);
			\draw [thick] (5,5) -- (5,0);
			\draw [thick] (5,0) -- (0,0);
			
			%	\node at (0,0) {\textbullet};
			
			%	\node at (1.25,3.75) {\textbullet};
			
			%	\node at (2.5,3.75) {\textbullet};
			
			%	\node at (3.75,0) {\textbullet};
			
			%	\node at (5,1.25) {\textbullet};
			
			\draw [thick] (3.75,0)--(5,1.25);
			\draw [thick] (2.5,3.75)--(3.75,0);
			\draw [thick] (1.25,3.75)--(2.5, 3.75);
			\draw [thick] (0,0)--(1.25, 3.75);
			
			% \draw [dashed] (1.5,0) -- (1.5,5);
			%  \draw [dashed] (2.75,0) -- (2.75,5);
			%   \draw [dashed] (4,0) -- (4,5);
			
			\draw [thick] (0,0) 
			node[below]{{$0$}} (0,0);
			
			\draw [thick] (1.25,0) 
			node[below]{{$\frac{1}{4}$}} (1.25,0);
			
			\draw [thick] (0,1.25) 
			node[left]{{$\frac{1}{4}$}} (0,1.25);
			
			\draw [thick] (2.5,0) 
			node[below]{{$\frac{1}{2}$}} (2.5,0);
			
			\draw [thick] (3.75,0) 
			node[below]{{$\frac{3}{4}$}} (3.75,0);
			
			\draw [thick] (0,3.75) 
			node[left]{{$\frac{3}{4}$}} (0, 3.75);
			
			\draw [thick] (5,0) 
			node[below]{{$1$}} (5,0);
			
			\draw [thick] (0,5) 
			node[left]{{$1$}} (0,5);
		\end{tikzpicture}
		\hspace{1cm}
		\begin{tikzpicture}[scale=0.9][domain=0:10]
			\draw [thick] (0,0) -- (0,5);
			\draw [thick] (0,5) -- (5,5);
			\draw [thick] (5,5) -- (5,0);
			\draw [thick] (5,0) -- (0,0);
			
			\node at (1.25,4.98) {\textbullet};
			
			\node at (2.5,2.5) {\textbullet};
			
			\node at (3.75,2.5) {\textbullet};
			
			%	\node at (0,0) {\textbullet};
			
			%	\node at (4.98,0) {\textbullet};
			
			\node at (2.5,1.25) {$\circ$};
			
			\node at (3.75,1.25) {$\circ$};
			
			\node at (1.25,3.75) {$\circ$};
			
			\draw [thick] (3.77,2.5)--(5,0);
			\draw [thick] (2.55,1.25)--(3.69,1.25);
			\draw [thick] (1.28,3.72)--(2.5, 2.53);
			\draw [thick] (0,0)--(1.25, 5);
			
			% \draw [dashed] (1.5,0) -- (1.5,5);
			%  \draw [dashed] (2.75,0) -- (2.75,5);
			%   \draw [dashed] (4,0) -- (4,5);
			
			\draw [thick] (0,0) 
			node[below]{{$0$}} (0,0);
			
			\draw [thick] (1.25,0) 
			node[below]{{$\frac{1}{4}$}} (1.25,0);
			
			\draw [thick] (0,1.25) 
			node[left]{{$\frac{1}{4}$}} (0,1.25);
			
			\draw [thick] (2.5,0) 
			node[below]{{$\frac{1}{2}$}} (2.5,0);
			
			\draw [thick] (0,2.5) 
			node[left]{{$\frac{1}{2}$}} (0,2.5);
			
			\draw [thick] (3.75,0) 
			node[below]{{$\frac{3}{4}$}} (3.75,0);
			
			\draw [thick] (0,3.75) 
			node[left]{{$\frac{3}{4}$}} (0, 3.75);
			
			\draw [thick] (5,0) 
			node[below]{{$1$}} (5,0);
			
			\draw [thick] (0,5) 
			node[left]{{$1$}} (0,5);
		\end{tikzpicture}
		\caption{Maps $f_1$ and $f_2$}
		\label{Fig1}
	\end{center}
	\end{figure}
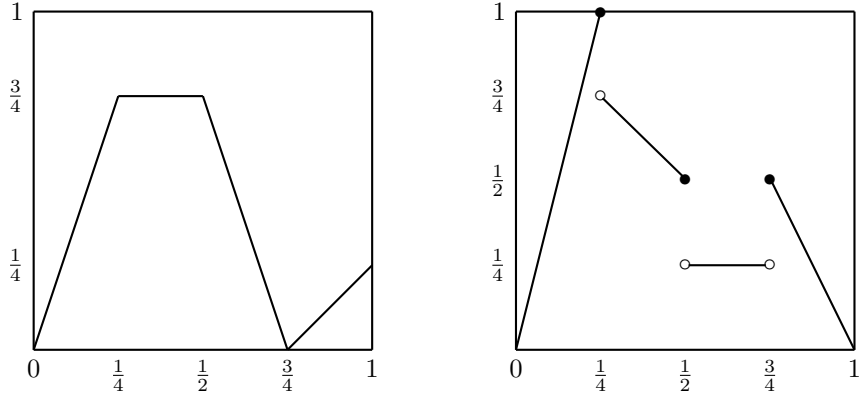
\end{example}

%\section{Constructing functions without iterative roots} \label{S3}
We now
%In this section we 
apply this tool to continuous self-maps on the spaces $[0,1]^m$, $\mathbb{R}^m$ and $S^1$ to prove Theorem 1.

\noindent {\it Proof 
	of Theorem \ref{NoRoot}
	for $[0,1]^m$ and $\mathbb{R}^m$.}
As seen in \cite{BG}, both $\mathcal{C}([0,1]^m)$ and $\mathcal{C}(\mathbb{R}^m)$ are metrizable. Further, 
the technique used to prove Theorem 3.8 (resp. Theorem 4.3) in \cite{BG}, which says that each open ball of $\mathcal{C}([0,1]^m)$ (resp. $\mathcal{C}(\mathbb{R}^m)$) has a map that does not have even discontinuous iterative square roots, is to exhibit a continuous self-map on $[0,1]^m$ (resp. $\mathbb{R}^m$) in each open ball of $\mathcal{C}([0,1]^m)$ (resp. $\mathcal{C}(\mathbb{R}^m)$) satisfying the hypotheses of one of the cases of Theorem 2.4 in \cite{BG}. Hence the results for $[0,1]^m$ and $\mathbb{R}^m$ follows from Theorem \ref{Tool}, because the hypotheses of Theorem 2.4 of \cite{BG} and Theorem \ref{Tool} coincide. \hfill $\square$

In order to prove Theorem \ref{NoRoot} for  $S^1$, it is convenient to
have some definitions and notation.
Given $z_0, z_1, \ldots, z_{k-1} \in S^1$, where $k\ge 2$ is fixed,
we define the {\it cyclic order} $\prec$ on $S^1$ by $z_0\prec z_1\prec \cdots \prec z_{k-1}$ if there exist
$t_1, t_2, \ldots, t_{k-1}\in \mathbb{R}$ such that $0< t_1<t_2<\cdots<t_{k-1}<1$
and $z_j=z_0e^{2\pi it_j}$ for $1\le j\le k-1$. 
%We can also use $z_k=z_0$ for notational convenience.
For any two distinct points $z_1, z_2 \in S^1$, with $z_1=e^{2\pi it_1}, z_2=e^{2\pi it_2}$ and  $0\leq t_1<t_2<t_1+1<2$, define
the  {\it open arc} $(z_1,z_2)$ and the {\it closed arc} $[z_1,z_2]$
by $(z_1,z_2):=\{z\in S^1: z_1\prec z\prec z_2\}:=
\{ e^{2\pi it}: t_1<t<t_2\}$ and $[z_1,z_2]:=(z_1,z_2)\cup \{z_1,z_2\}$.
On the other hand, let $(z_2,z_1):=\{e^{2\pi it}\in S^1: t_2< t<
1+t_1\}$ and $[z_2, z_1]:=(z_2, z_1)\cup \{z_2, z_1\}$. Additionally, if $|z_2-z_1|<2$, then only one of the arcs $[z_1,z_2]$ and $[z_2,z_1]$ is a minor (smaller) arc of $S^1$ as none of them are semicircles.

By a {\it partition} $P$ of $S^1$ we mean a finite set of points $z_0, z_1,\ldots, z_{k-1} \in S^1$ for some $k\ge 2$
such that $0\prec z_0\prec z_1\prec \cdots \prec z_{k-1}$  and
\begin{eqnarray*}
	S^1=\bigcup_{j=0}^{k-1}[z_j , z_{j+1}],
\end{eqnarray*}
where $z_k=z_0$.
In this case, we refer to $J_j:=[z_j, z_{j+1}]$ as the $j$-th arc of the partition $P$ for $0\leq j\leq k-1$.  Let $J(P)$ denote the
collection $\{J_j: 0\leq j\leq k-1\}$ of arcs of $P$ and $J_j^0$ 
the interior $(z_j, z_{j+1})$ of $J_j$ for $0\leq j\leq k-1$. We say that a partition $Q$ of $S^1$ a {\it refinement} of $P$ if $Q\supseteq P$.

Let $J=[z_1, z_2]$ be an arc in $S^1$, where $z_j=e^{2\pi it_j}$ for
$j\in \{1,2\}$ such that $0\le t_1<t_2<t_1+1<2$, and $w_j=e^{2\pi
	is_j} \in S^1$ for  $j\in \{1,2\}$, where $0\le s_1, s_2<1$.  Then
we can define an {\it affine linear map} $f:J\to S^1$ such that
$f(z_j)=w_j$ for all $j\in \{1,2\}$ by
\begin{eqnarray}\label{affine-map}
	f(e^{2\pi i\{\alpha t_{2} +(1-\alpha)t_{1}\}})=
	e^{2\pi i\{\alpha s_{2} +(1-\alpha)s_{1}\}},~\forall \alpha\in [0,1]
\end{eqnarray}
or
\begin{eqnarray}\label{affine-map2}
	f(e^{2\pi i\{\alpha t_{2} +(1-\alpha)t_{1}\}})=e^{2\pi i\{\alpha (s_{2}-1) +(1-\alpha)s_{1}\}},~\forall \alpha\in [0,1]
\end{eqnarray}
according as $f$ maps the arc $[z_1, z_2]$ onto $[w_1,w_2]$ or
$[w_2, w_1]$, respectively. Each affine linear map on $J$ is clearly
continuous and open.  A self-map $f$ on $S^1$ is said to be {\it
	piecewise affine linear} if there exists a partition $P:z_0\prec
z_1\prec \cdots \prec z_{k-1}$ of $S^1$ such that the restriction of
$f$ to the arc $J_j$ is  affine linear for $0\le j\le k-1$. In this
case we say that $f$ is supported by the partition $P$. 
%In the following we would assume $|w_j-w_{j+1}|<\frac{1}{2}$ for every $j$. This ensures that the two arcs $[w_j, w_{j+1}]$ and $[w_{j+1} ,w_j]$ can not have equal length. One of them is shorter than the other.
A piecewise affine linear self-map $f$ on $S^1$ supported on a partition $P:z_0\prec
z_1 \prec \cdots \prec z_{k-1}$ with values $f(z_j)=w_j$ for
$0\leq j\leq k-1$, where $|w_{j+1}-w_j|<2$ for all $0\leq j\leq k-1$ with $w_k=w_0$, 
is said to be {\it admissible} if $f$ maps the arc $[z_j,
z_{j+1}]$ to the minor arc among $[w_j, w_{j+1}]$ and
$[w_{j+1}, w_j]$ for all $0\leq j\leq k-1$.  Each admissible piecewise affine linear map on $S^1$ is clearly
continuous. 

\noindent {\it Proof of Theorem \ref{NoRoot} for $S^1$.}
Consider $\mathcal{C}(S^1)$ in its canonical metric
$\rho(f,g)=\sup\{|f(z)-g(z)|:z\in S^1\}$, the uniform metric, and
let  $B_\epsilon(h):=\{f\in \mathcal{C}(S^1): \rho(f,h)<\epsilon\}$
be any ball in $\mathcal{C}(S^1)$ of radius $\epsilon\in (0,1)$
centered at $h$. 

Since $h$ is uniformly continuous on $S^1$, there exists  a $\delta$ with $0<\delta <1/4$ such that
\begin{eqnarray}\label{1}
	|h(z)-h(z')| <\frac{\epsilon }{10} %~~\text{whenever}~~s,t\in [0,1)~~\text{with}~~|t-s|<\delta.
\end{eqnarray}
whenever $z,z'\in S^1$ with $|z-z'|<\delta$. Consider a
partition $P:z_0\prec z_1\prec \cdots \prec z_{k-1}$ of $S^1$ for some $k\ge 2$ with  $|z_{j+1}-z_{j}|<\delta/2$
for $0\le j\le k-1$,
where $z_j=z_0e^{2\pi it_j}$  for $0\le j\le k-1$ such that
$0\le t_0<t_1< t_2<\cdots <t_{k-1}<1$. Let $J_j:=[z_{j}, z_{j+1}]$  for $0\le j\le k-1$.  

Our first step is to choose an admissible affine linear map
$f_0$ in $B_\epsilon(h)$ supported on $P$ satisfying that {\bf (i)} $f_0|_{J_j}$ is injective for $0\le j\le k-1$; and {\bf (ii)} $f_0|_{J_j}\ne {\rm id}$ for $0\le j\le k-1$.
We can easily choose distinct
points $w_j:=e^{2\pi is_j}$ in $S^1$, where $0\le  s_0, s_1, \ldots,
s_{k-1}<1$, such that $w_j\ne z_j$ for $0\leq j\leq k-1$ and
\begin{eqnarray}\label{2}
	|h(z_j)-w_j|<\frac{\epsilon }{20}
	%~~\text{for all}~~0\le j\le k.
\end{eqnarray}
for $0\le j\le k-1$. Then, by \eqref{1} and \eqref{2}   we have
\begin{eqnarray}
	|w_{j+1}-w_j|&\le &  |w_{j+1}-h(z_{j+1})|+|h(z_{j+1})-h(z_j)|+|h(z_j)-w_j|\nonumber \\
	&<& \frac{\epsilon}{20}+\frac{\epsilon }{10}+\frac{\epsilon}{20}=\frac{\epsilon}{5}, \label{5}
\end{eqnarray}
%	$|w_{j+1}-w_j|<\epsilon/5$ 
implying that only one of the arcs $[w_j,w_{j+1}]$ and $[w_{j+1},w_j]$ is a minor arc of $S^1$ for  $0\le j\le k-1$, where $w_k=w_0$. Let
$f_0$ be the admissible piecewise affine linear map on $S^1$ supported on $P$
with values $f_0(z_j)=w_j$ for $0\le j\le k-1$.  Then, clearly $f_0$ is a continuous self-map on $S^1$.
Also,
since $w_{j+1}\ne w_{j}$ and $w_j \ne z_j$ for  $0\le j \le k-1$, we see that $f_0$ satisfies {\bf (i)} and {\bf (ii)}.  
Further, if $z\in J_j$, where $0\le j\le k-1$, then by using \eqref{1}, \eqref{2}, \eqref{5} and the definition of $f_0$, we have
\begin{eqnarray}
	|f_0(z)-h(z)|&\le & |f_0(z)-w_j|+|w_j-h(z_j)|+|h(z_j)-h(z)| \nonumber \\
	&< &\frac{\epsilon }{5}+\frac{\epsilon}{20}+\frac{\epsilon }{10} <\frac{\epsilon}{2}, \label{4}
\end{eqnarray}
%	$$|h(x)-f_0(x)|< \frac{\epsilon}{4?}, ~~\forall x\in S^1.$$
proving that $f_0\in B_\epsilon(h)$.

Next, we choose an arc $J:=[u_0,u_1]$ in $S^1$ so small that {\bf (a)} $J$ is contained in the interior of $R(f_0)$; {\bf (b)} $J\subseteq J_r^0$
for some $0\le r\le k-1$; and {\bf (c)} $f_0(J)\subseteq J_{r'}^0$ for some
$0\le r'\le k-1$.  This is possible by virtue of {\bf (i)} and {\bf (ii)}, because $R(f_0)$  has non-empty interior by {\bf (i)}. Now consider $f_0$ on the arc $J$. Since $f_0|_J\ne {\rm id}$ by {\bf (ii)}, we can choose an $a\in J$ such that $f_0(a)\ne a$. This allows us to replace $J$ with a smaller arc around $a$, which we again denote by $J$, satisfying  that {\bf (d)} $f_0(J)\cap J=\emptyset$ and thus {\bf (e)} $f_0^{-1}(J)\cap J=\emptyset$, in addition to {\bf (a)}, {\bf (b)} and {\bf (c)}. Now choose an arc $K:=[y_0,y_1]$  contained in $J^0$ such that $y_0\ne y_1$. Since $K\subseteq J_r^0$ by {\bf (a)},  and $f_0$ is affine linear and non-constant on $J_r$, we have $f_0(K)=[x_0, x_1]$ for some $x_0\ne x_1$ in $J_{r'}^0$.
%this also allows us to replace the partition $P$ with a finer partition if necessary, again denoted by $P$, and assume that $r'\ne r$.
%%this also
%allows us to replace the partition $P$ with a finer partition if
%needed and assume $s\neq r.$
%Now,
%Additionally, 
%since $K\subseteq J_r^0$ by {\bf (a)},  and $f_0$ is affine linear and non-constant on $J_r$, we have $f_0(K)=[y_0, y_1]$ for some $y_0\ne y_1$ in $J_{r'}^0$, where $K:=[x_0,x_1]$ is a chosen arc in $S^1$ contained in $J^0$ with $x_0\ne x_1$.
 Further, since $[x_0,x_1]\subseteq J_{r'}^0$ by {\bf (c)}, and $f_0$ is affine linear and non-identity on $J_{r'}$, we have either $f_0(x_0)\neq x_0$ or $f_0(x_1)\neq x_1$. 
Without loss of generality,  we assume that $f_0(x_0)\neq x_0$.

Now consider the refinement  $Q:z_0\prec \cdots \prec z_r\prec u_0
\prec y_0 \prec y_1 \prec u_1\prec z_{r+1}\prec \cdots \prec z_{k-1}$ of $P$ and let
$f$ be the admissible piecewise affine linear map on $S^1$  supported on
$Q$ with values $f(z_j)= f_0(z_j)= w_j$ for  $0\le j\le k-1$, $f(u_0)=f_0(u_0)$, $f(u_1)=f_0(u_1)$ and $f(y_0)=f(y_1)=x_0$. Then {\bf (f)} $f(z)=f_0(z)$ for all $z\in S^1\setminus J$, and $f$ has the
constant value $x_0$ on $K$,  so that  $f_0$ is modified only on $J$. 
Further, given any $z\in J$,
by {\bf (b)} and \eqref{5} we have
\begin{eqnarray*}
	|f(z)-f_0(z)|
	% &<&\sum_{i=0}^{m-1}\alpha_{j_i}\|f(u_{j_i})-\tilde{f}_0(u_{j_i})\|_\infty
	< |w_{r+1}-w_r|
	<  \frac{\epsilon}{5},
\end{eqnarray*}
implying by \eqref{4} that
\begin{eqnarray*}
	|f(z)-h(z)|\le |f(z)-f_0(z)|+|f_0(z)-h(z)|
	< \frac{\epsilon}{5}+\frac{\epsilon}{2}<\frac{4\epsilon}{5}.
\end{eqnarray*}
Therefore $f\in B_\epsilon(h)$
because we already know by {\bf (f)} and \eqref{4} that $|f(z)-h(z)|=|f_0(z)-h(z)|<\epsilon/2$ for all $z\in S^1\setminus J$.

Since $y_0\in K\subseteq J$,  by {\bf (d)} and {\bf (f)} we see that $f(x_0)=f_0(x_0)\ne x_0$. Also, since $f=f_0$ on $f_0^{-1}(K)$ by {\bf (e)} and {\bf (f)}, we have $f^{-2}(x_0)
%\supseteq f^{-1}(f^{-1}(x_0))
\supseteq f^{-1}(K)
\supseteq f_0^{-1}(K)$, implying that $f^{-2}(x_0)$ is infinite, because  $K$ is contained in the interior of $R(f_0)$ by {\bf (a)}. Further, $f|_I$ is injective for all arcs $I\in J(Q)$ with $I\ne K$ so that $f^{-1}(z)$ is finite for all $z\ne x_0$ in $S^1$.
Thus $f$ satisfies all the conditions of {\bf (C2)} of Theorem \ref{Tool}, proving that it does not have even discontinuous iterative roots of order $n\ge 2$ on $S^1$.  \hfill $\square$

\section{Some Remarks about the Tool}\label{S4}
%Finally, 
%n this section we make some interesting remarks about our  tool Theorem \ref{Tool} with examples.
%We conclude the paper with some intriguing insights about our tool Theorem \ref{Tool} with examples.
%We conclude the paper with some additional comments about our tool Theorem \ref{Tool}, justified with examples
We conclude the paper with some additional comments about our tool Theorem \ref{Tool}.

%\noindent {\bf Remarks:}	
\begin{remark}
	$f^{-2}$ considered in Theorem \ref{Tool} cannot be replaced by $f^{-1}$, as seen from Example 2.5 in \cite{BG}.
\end{remark} 

\begin{remark} 
	Consider the set $X=\{x_j:-8\le j\le 0\}$
	and let $f$ be the constant map on $X$ defined by $f(x_j)=x_0$ for all $-8\le j\le 0$.
	 Then $f$ satisfies the conditions in {\bf (C1)} of Theorem \ref{Tool} with $N=2$. However, $f(x_0)=x_0$ and $f$ has an iterative square root $g$ on $X$ given by 
\begin{eqnarray*}
	g(x_j)=\left\{\begin{array}{cll}
		x_{j+4}&\text{if}&-8\le j\le -5,\\
		x_0&\text{if}& -4\le j\le 0
	\end{array}\right.
\end{eqnarray*}
(see Figure \ref{Fig2}). Therefore the assumption that $f(x_0)\ne x_0$ made in Theorem \ref{Tool} cannot be dropped.
\begin{figure}[h]
	\centering
	\includegraphics[width=6cm,height=4.5cm]{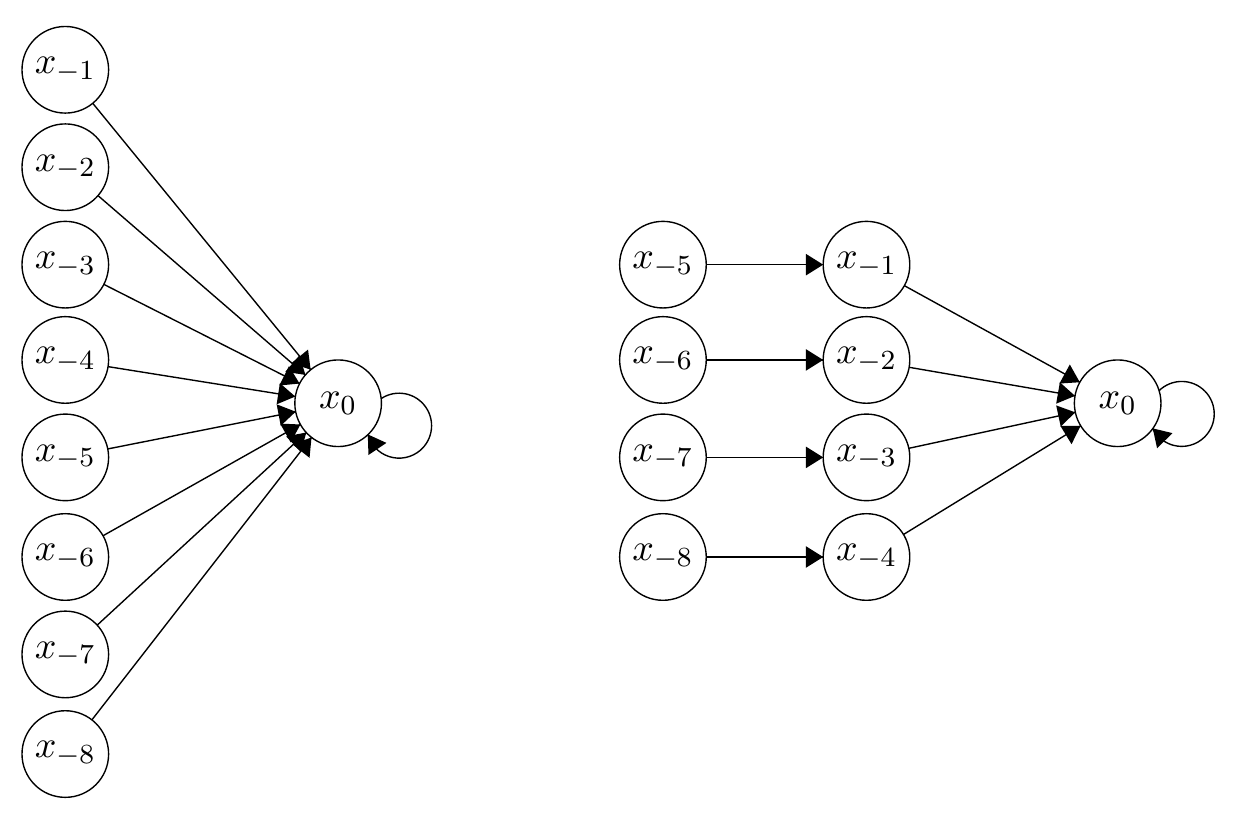}
	\caption{Maps $f$ and $g$}
	\label{Fig2}
\end{figure}
\end{remark}
\begin{remark}
	Let $X=\{x_j:-16\le x\le 2\}$ and  $f:X\to X$ be
  the map defined by 
\begin{eqnarray*}
	f(x_j)=\left\{\begin{array}{clc}
		x_2&\text{if}&-16\le j\le -9,\\
		x_0&\text{if}&-8\le j\le -1,\\
		x_1&\text{if}&j=0,\\
		x_2&\text{if}&j=1,\\
		x_0&\text{if}&j=2.
	\end{array}\right.
\end{eqnarray*}
Then $f(x_0)\ne x_0$ and $\#f^{-2}(x_0)>2^3$. However, $\#f^{-1}(x_2)>2$ and $f$ has an iterative square root $g$ on $X$ given by 
\begin{eqnarray*}
	g(x_j)=\left\{\begin{array}{clc}
		x_0&\text{if}&-16\le j\le -9,\\
		x_{j-8}&\text{if}&-8\le j\le -1,\\
		x_2&\text{if}&j=0,\\
		x_0&\text{if}&j=1,\\
		x_1&\text{if}&j=2
	\end{array}\right.
\end{eqnarray*}
(see Figure \ref{Fig3}). Therefore the condition that $\#f^{-1}(x)\le N$ for all $x\ne x_0$ considered in {\bf (C1)} of Theorem \ref{Tool} cannot be relaxed.  Similarly, the condition that $f^{-1}(x)$ is finite (resp.  countable) for all $x\ne x_0$ considered in {\bf (C2)} (resp. {\bf (C3)}) of Theorem \ref{Tool} cannot be weakened.
\begin{figure}[h]
	\centering
	\includegraphics[width=9cm,height=4.5cm]{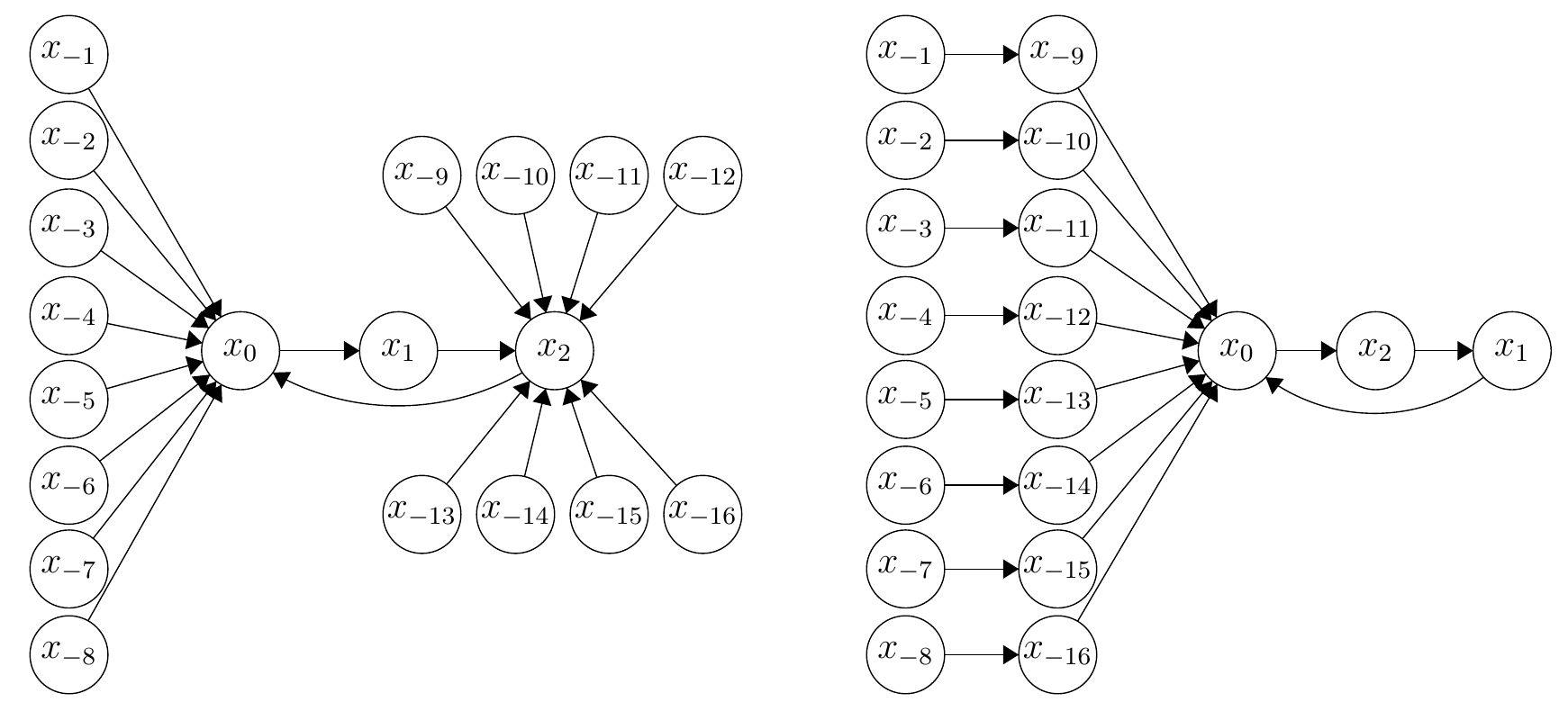}
	\caption{Maps $f$ and $g$}
	\label{Fig3}
\end{figure}
\end{remark}
\begin{remark}
The bound $N^3$ considered in {\bf (C1)}
of Theorem \ref{Tool} is optimal as seen from the following examples. In fact, if $N=1$, then a $3$-cycle $f$ on $X=\{x_0, x_1, x_2\}$ satisfying that $f(x_0)\ne x_0$,   $\#f^{-2}(x_0)=1$ and $\#f^{-1}(x_j)=1$ for all $j\ne 0$ has an iterative square root. If $N=2$, take $X=\{x_j: j=-12, -11, \ldots\}\cup \{y_j: j=-6,-5,\ldots\}$ and consider the map $f:X\to X$ defined by
\begin{eqnarray*}
	f(x_j)=\left\{\begin{array}{cll}
		x_{-4}&\text{if}&j=-12, -11,\\
		x_{-3}&\text{if}&j=-10, -9,\\
		x_{-2}&\text{if}&j=-8, -7,\\
		x_{-1}&\text{if}&j=-6, -5,\\
		x_{0}&\text{if}&j=-4,-3,-2,-1,\\
		x_{j+1}&\text{if}& j\ge 0
	\end{array}\right.
\end{eqnarray*}
and
	\begin{eqnarray*}
	f(y_j)=\left\{\begin{array}{cll}
		y_{-2}&\text{if}&j=-6, -5,\\
		y_{-1}&\text{if}&j=-4, -3,\\
		y_{0}&\text{if}&j=-2, -1,\\
		y_{j+1}&\text{if}&j\ge 0
	\end{array}\right.
\end{eqnarray*}
(see Figure \ref{Fig4}).  Then $f(x_0)\ne x_0$, $\#f^{-2}(x_0)=8$ and $\#f^{-1}(x)\le 2$ for all $x\ne x_0$. However, $f$ has an iterative square root $g$ on $X$ given by
\begin{eqnarray*}
	g(x_j)=\left\{\begin{array}{cll}
		y_{-6}&\text{if}&j=-12, -11,\\
		y_{-5}&\text{if}&j=-10, -9,\\
		y_{-4}&\text{if}&j=-8, -7,\\
		y_{-3}&\text{if}&j=-6, -5,\\
		y_{-2}&\text{if}&j=-4,-3,\\
		y_{-1}&\text{if}&j=-2,-1,\\
		y_j&\text{if}& j\ge 0
	\end{array}\right.
\end{eqnarray*}
and
	\begin{eqnarray*}
	g(y_j)=\left\{\begin{array}{cll}
		x_{j+2}&\text{if}&j=-6,-5,-4,-3\\
		x_0&\text{if}&j=-2,-1,\\
		x_{j+1}&\text{if}&j\ge 0
	\end{array}\right.
\end{eqnarray*}
(see Figure \ref{Fig5}). Similar examples can be given when $N>2$.  
\begin{figure}[h]
	\centering
	\includegraphics[width=9cm,height=4.5cm]{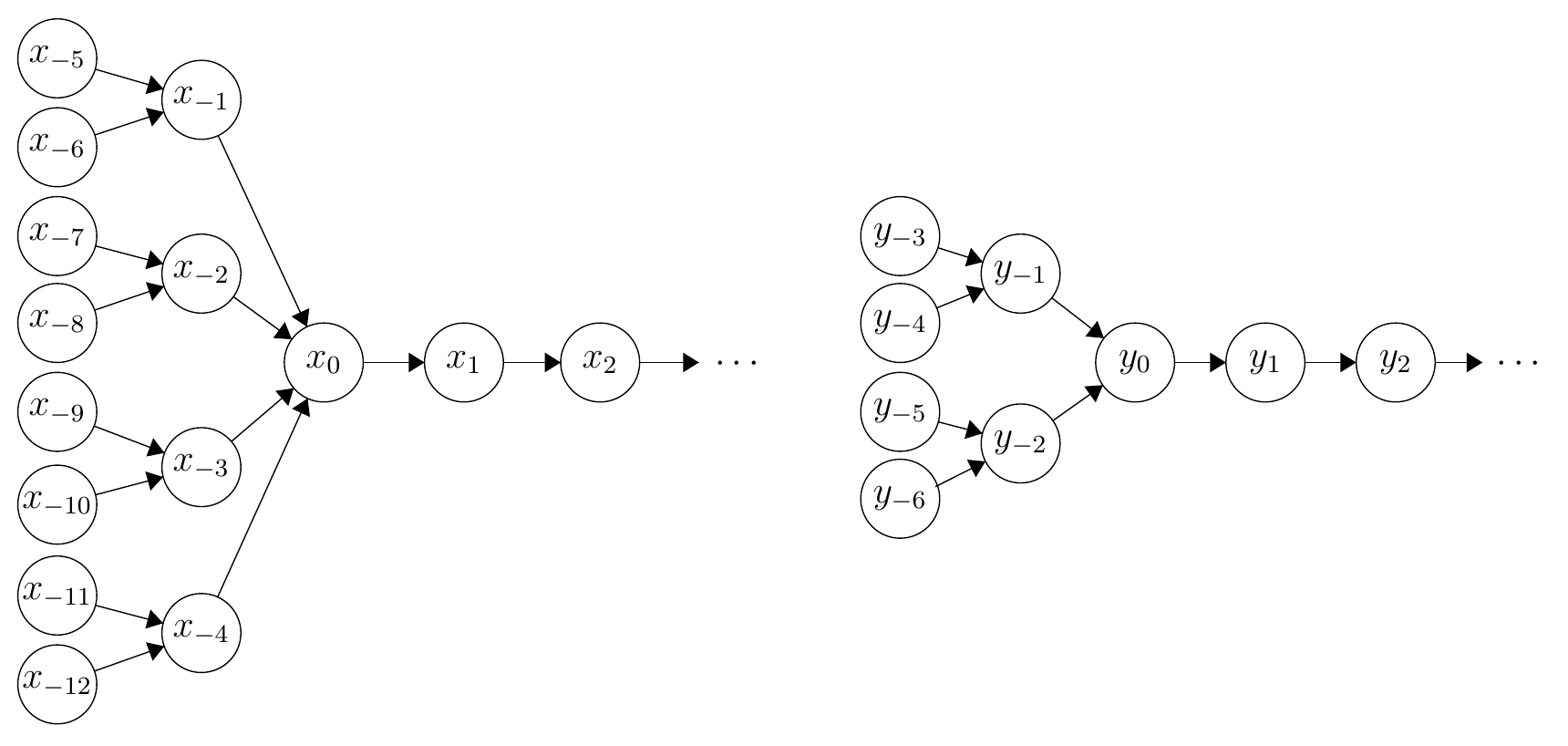}
	\caption{The map $f$}
	\label{Fig4}
\end{figure}
\begin{figure}[h]
	\centering
	\includegraphics[width=7.5cm,height=4.5cm]{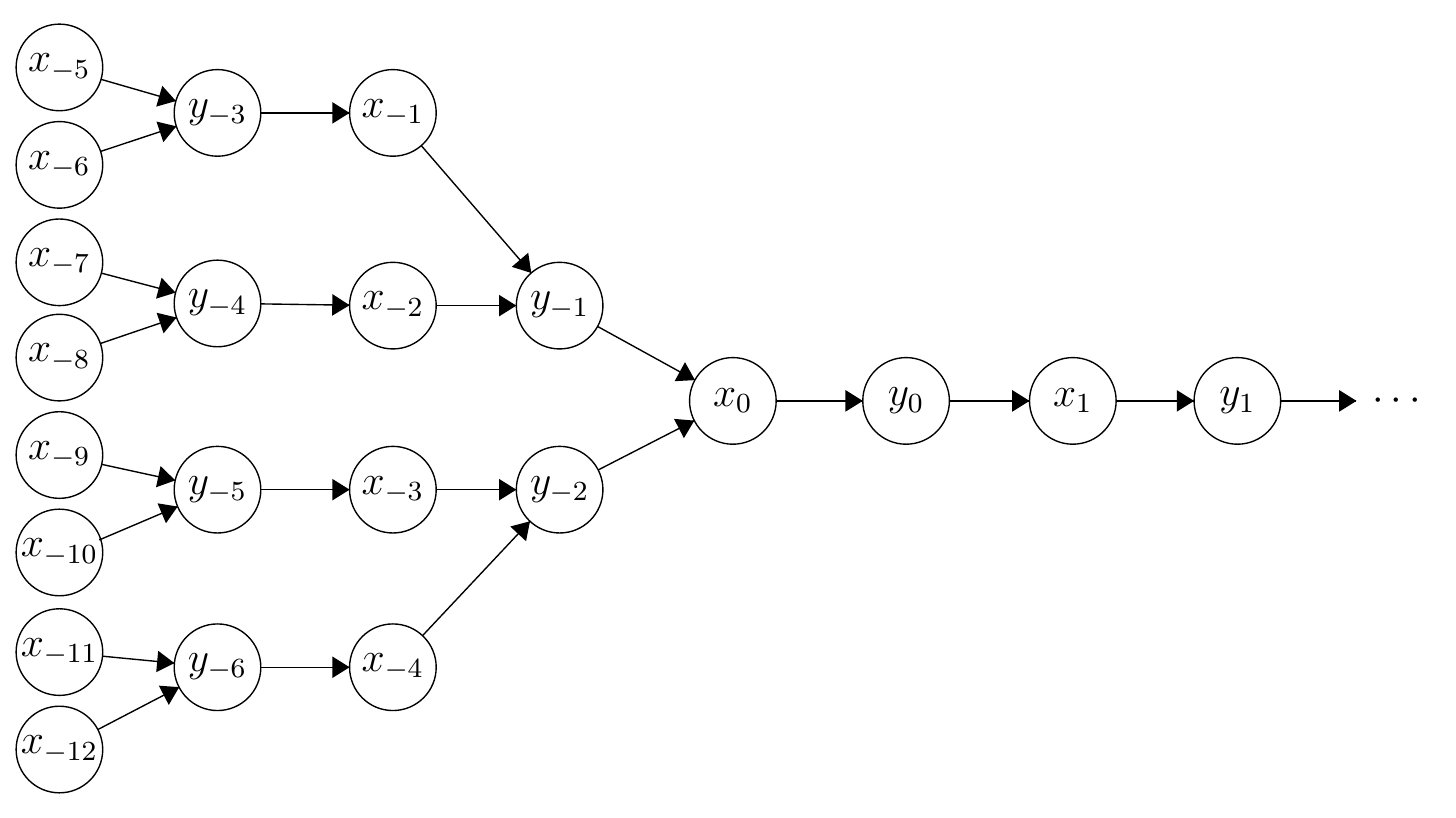}
	\caption{The map $g$}
	\label{Fig5}
\end{figure}
\end{remark}
\begin{remark}\label{R5}
 By comparing the three cases, one would expect the natural analogue of Theorem \ref{Tool} to be valid  in other contexts where  $f^{-2}(x_0)$ is ``big'' and  $f^{-1}(x)$ is ``small'' for all $x\ne x_0$. It is indeed true in the context of cardinal numbers under the assumption of the axiom of choice.
%  More precisely, assuming the axiom of choice, we can prove that the result in Theorem \ref{Tool} is also valid when $\#f^{-2}(x_0)\ge \aleph_\alpha$ and $\#f^{-1}(x)<\aleph_\alpha$ for each $\alpha$, where $\aleph_\alpha$'s are precisely the cardinal numbers indexed by the ordinal numbers $\alpha$ and $\le$ is the order among the cardinal numbers $\aleph_\alpha$.
More precisely, we have the following result, where $\aleph_\alpha$'s are precisely the infinite cardinal numbers indexed by the ordinal numbers $\alpha$ and $\le$ is the order among the cardinal numbers $\aleph_\alpha$.
\end{remark}
\begin{theorem}\label{Tool-cardinal}
		Let $f$ be a self-map on a non-empty set $X$ such that $f(x_0)\ne x_0$ for some $x_0\in X$, $\#f^{-2}(x_0)\ge \aleph_\alpha$ and $\#f^{-1}(x)<\aleph_\alpha$ for all $x\ne x_0$ in $X$.
	Then $f$ has no iterative roots of order $n\ge 2$ on $X$.
\end{theorem}
\begin{proof}
%	The proof of  {\bf (C2)} of Theorem \ref{Tool} repeatedly uses the  fact that a	finite union of finite sets is finite. The proof of this result 
It is similar to that of  {\bf (C2)} of Theorem \ref{Tool}, using the result that a union of a collection of cardinality $\aleph_\alpha$ of sets of cardinality $\aleph_\alpha$ has cardinality $\aleph_\alpha$. 
\end{proof}

   However, the natural analogue of Theorem \ref{Tool} is no longer true in the context of measures. More precisely, the implication of Theorem \ref{Tool} is no longer true when $f$ is a measurable self-map on a measure space $X$ with $f(x_0)\ne x_0$ for some $x_0\in X$ such that $f^{-2}(x_0)$ has positive measure and $f^{-1}(x)$ has measure zero for all $x\ne x_0$. This can be seen from the following example. 
\begin{example}\label{Ex4}
	{\rm 
	
%	Let  $C_1, 	C_2, C_3$ be Cantor ternary sets and $\hat{C}$ be an $\epsilon$-Cantor set (see \cite[pp.140-141]{A-B-1998}) in four disjoint compact intervals of $\mathbb{R}$ with $0<\epsilon<1$.
Let $C_1, C_2, C_3$ and $\hat{C}$ be  Cantor spaces (i.e., topological spaces homeomorphic to the Cantor ternary set (\cite[p.14]{Kechris}) or equivalently those that  are non-empty, perfect, compact,  metrizable, and zero-dimensional (\cite{Brouwer}, \cite[p.35]{Kechris})) in four disjoint compact intervals of $\mathbb{R}$, where $C_1, C_2$ and $C_3$ have  (Lebesgue) measure zero and $\hat{C}$ has positive measure (consider $\hat{C}$ to be an $\epsilon$-Cantor set with $0<\epsilon<1$ (see \cite[pp.140-141]{A-B-1998}) for example). 
%Then, by a known result in \cite{Brouwer}, which says that any two non-empty, perfect, compact, totally disconnected metrizable spaces are homeomorphic to each other,  we know that  $C_1, C_2, C_3$ and $\hat{C}$ are mutually homeomorphic.
	Let $x_0, x_1,x_2$ and $x_3$ be distinct real numbers not in $\hat{C}\cup C_1\cup C_2\cup C_3$. Consider the (Lebesgue) measurable map $g:X\to X$ defined by the sequence 
	\begin{eqnarray*}
		\hat{C}\xrightarrow{\phi_1}C_1\xrightarrow{\phi_2}C_2\xrightarrow{\phi_3}C_3\to \{x_0\}\to \{x_1\}\to \{x_2\}\to \{x_3\}\to \{x_0\},
	\end{eqnarray*}
	where $X:=\hat{C}\cup C_1\cup C_2\cup C_3\cup\{x_0, x_1, x_2,x_3\}$, and $\phi_1:\hat{C}\to C_1$, $\phi_2:C_1\to C_2$ and $\phi_3:C_2\to C_3$ are homeomorphisms. Then the map $f:=g^2$, which is given by the sequences
	\begin{eqnarray*}
		\hat{C}\xrightarrow{\phi_2\circ\phi_1}C_2\to \{x_0\}\to  \{x_2\}\to \{x_0\}\quad \text{and}\quad 	C_1\xrightarrow{\phi_3\circ\phi_2}C_3 \to \{x_1\}\to  \{x_3\}\to \{x_1\},
	\end{eqnarray*}
	is  measurable on $X$.
Moreover,  $f(x_0)\ne x_0$, $f^{-2}(x_0)=\hat{C}$ has positive measure and $f^{-1}(x)$ has measure zero for all $x\ne x_0$ in $X$. It is worth noting that $f$ maps measure zero sets to measure zero sets. 
%	This attribute, however, is irrelevant here as we can construct similar maps that carry sets of measure zero to sets of positive measure.
%Thus, this example in particular shows that the Theorem \ref{Tool} is no longer true even for measurable maps satisfying this additional condition.
Thus, this example in particular shows that the natural analogue of Theorem \ref{Tool} for measurable functions, where sets of positive measure are considered large compared to sets of  measure zero, fails even for functions which send sets of measure zero to sets of measure zero.
	}
\end{example}

\bibliographystyle{amsplain}

\begin{thebibliography}{10}

%	\bibitem{Abel}
%N. H. Abel, Oeuvres Complètes, T.II, {\it Christiania}, (1881), 36--39.

\bibitem{A-B-1998}
C. D. Aliprantis, O. Burkinshaw, {\it Principles of real analysis}, Third Edition, Academic Press, Inc., San Diego, CA, 1998. 

%\bibitem{arens1946} R. F. Arens,
%A topology for spaces of transformations, {\it Ann. of Math.},  47 (1946), 3, 480--495.

%\bibitem{Babbage1815}
%C. Babbage, Essay towards the calculus of functions, {\it Philos. Trans.}, (1815), 389--423.
%
%\bibitem{Baron-Jarczyk}
%K. Baron, W. Jarczyk, Recent results on functional equations in a single variable,
%perspectives and open problems, {\it Aequationes Math.},  61 (2001), 1--48.

%\bibitem{BDR}
%B. V. R. Bhat, S. De, N. Rakshit, A caricature of
%dilation theory, {\it Adv. Oper. Theory}, 6 (2021), 4, Paper No. 63, 20 pp. %arXiv:2004.09255.

\bibitem{BG}
B. V. R. Bhat, C. Gopalakrishna, Iterative square roots of functions, {\it Ergodic Theory Dynam. Systems}, (2022), pp. 1--27. %\url{https://arxiv.org/abs/2105.02171}.
\url{https://doi.org/10.1017/etds.2022.35}

%\bibitem{blokh-coven}
%A. Blokh, E. Coven, M. Misiurewicz, Z. Nitecki, Roots of continuous piecewise monotone
%maps of an interval, {\it Acta Math. Univ. Comenian. (N.S.)},  60 (1991), 3--10.

\bibitem{Blokh}
A. M. Blokh, The set of all iterates is nowhere dense in
{$C([0,1],[0,1])$}, {\it Trans. Amer. Math. Soc.},  333 (1992), 2,  787--798.

%\bibitem{Bodewadt}
%U. T. B\"{o}dewadt,  Zur Iteration reeller funktionen, {\it Math. Z.}, 49  (1944), 497--516.

%\bibitem{Bogatyi}
%S. Bogatyi, On the nonexistence of iterative roots, {\it Topology Appl.}, 76
%(1997), 97--123.

\bibitem{Brouwer}
L. E. J. Brouwer,  On the structure of perfect sets of points, Proc. Koninklijke Akademie van Wetenschappen, 12 (1910), 785--794.

%\bibitem{deo}
%S. Deo, {\it Algebraic topology}, Hindustan Book Agency, New Delhi, 2018.

%\bibitem{dugundji}
%J. Dugundji, An extension of Tietze's theorem, {\it Pacific J. Math.},  1 (1951), 353--367.
%
%\bibitem{Edgar}
%G. A. Edgar, Fractional iteration of series and transseries, {\it Trans. Amer. Math. Soc.},  365 (2013), 11,  5805--5832.
%
\bibitem{fort1955}
M. K. Fort Jr, The embedding of homeomorphisms in flows, {\it Proc. Amer. Math. Soc.},  6 (1955),
960--967.

\bibitem{Humke-Laczkovich1989}
P. D. Humke, M. Laczkovich, The Borel structure of iterates of continuous functions, {\it Proc. Edinburgh Math. Soc.},  32 (1989), 483--494.

\bibitem{Humke-Laczkovich}
P. D. Humke, M. Laczkovich, Approximations of continuous functions by squares, {\it Ergodic Theory Dynam. Systems}, 10 (1990), 2, 361--366.

%\bibitem{Iannella}
%N. Iannella and L. Kindermann, Finding iterative roots with a spiking
%neural network, {\it Inform. Process. Lett.}, 95 (2005), 545--551.
%
%\bibitem{Issacs}
%R. Isaacs, Iterates of fractional order, {\it Canad. J. Math.},  2 (1950), 409--416.
%
%\bibitem{Jarczyk-Zhang}
%W. Jarczyk, W. Zhang, Also set-valued functions do not like
%iterative roots, {\it Elemente Math.}, 62 (2007),  1--8.

\bibitem{Kechris}
A. S. Kechris, {\it Classical descriptive set theory}, Graduate Texts in Mathematics, vol. 156, Springer-Verlag, New York, 1995.

%
%\bibitem{Kindermann}
%L. Kindermann, Computing iterative roots with neural networks,
%{\it Proc. Fifth Conf. Neural Info. Processing} 2 (1998), 713--715.
%
%\bibitem{Konigs}
%G. K\"{o}nigs, Recherches sur les int\'{e}grales de certaines \'{e}quations fonctionnelles,
%{\it Ann. Ecole Norm. Sup.} (3) (1884), Suppl.  1, 3--41.
%
%\bibitem{kuczma1961}
%M. Kuczma, On the functional equation $\phi^n(x) = g(x)$, {\it Ann. Polon. Math.},  11 (1961), 161--175.
%
%\bibitem{Kuczma1968}
%M. Kuczma, {\it Functional Equations in a Single Variable}, Polish Scientific, Warsaw, 1968.
%
%\bibitem{Kuczzma1969}
%M. Kuczma, Fractional iteration of differentiable functions, {\it Ann. Pol. Math.}, 22 (1969/70), 217--227.
%
\bibitem{kuczma1990}
M. Kuczma, B. Choczewski, R. Ger,  {\it Iterative functional equations}, volume
32 of Encyclopedia of Mathematics and its Applications. Cambridge University
Press, Cambridge, 1990.
%
%\bibitem{Lesniak2002}
%Z. Le\'{s}niak, On fractional iterates of a homeomorphism of the plane, {\it Ann. Polon. Math.},   79 (2002), 2, 129--137.
%
%\bibitem{Lesniak2010}
%Z. Le\'{s}niak, On fractional iterates of a Brouwer homeomorphism embeddable in a flow, {\it J. Math.
	%	Anal. Appl.},  366 (2010), 310--318.
%
%
%
%\bibitem{li2009}
%L. Li, J. Jarczyk, W. Jarczyk, W. Zhang, Iterative roots of mappings with a unique set-value point,
%{\it Publ. Math. Debrecen},  75 (2009), 203--220.
%
%\bibitem{LiYangZhang2008}
%L. Li, D. Yang, W. Zhang, A note on iterative roots of PM functions, {\it J. Math. Anal. Appl.},  341  (2008), 2, 1482--1486.
%
%\bibitem{li2012}
%L. Li, W. Zhang, Construction of usc solutions for a multivalued iterative equation of order
%$n$, {\it Results Math.},  62 (2012), 203--216.
%
%\bibitem{LiZhang2018}
%L. Li, W. Zhang, The question on characteristic endpoints for iterative roots of PM functions, {\it J. Math. Anal. Appl.},   458  (2018), 1,  265--280.
%
%\bibitem{Lin2014}
%Y. Lin, Existence of iterative roots for the sickle-like functions, {\it J.
	%	Inequal. Appl.}, (2014), Paper No. 204, 23 pp.
%
%\bibitem{Lin-Zeng-Zhang2017}
%Y. Lin, Y. Zeng, W. Zhang, Iterative roots of clenched single-plateau
%functions, {\it Results Math.}, 71 (2017),  1-2, 15--43.
%
%\bibitem{LiuJarczykZhang2012}
%L. Liu, W. Jarczyk, L. Li, W. Zhang, Iterative roots of piecewise monotonic functions of nonmonotonicity
%height not less than $2$, {\it Nonlinear Anal.},   75 (2012), 1,  286--303.
%
%\bibitem{LiuLiZhang2018}
%L. Liu, L. Li, W. Zhang, Open question on lower order iterative roots for PM functions, {\it J. Difference Equ. Appl.},   24    (2018), 5, 825--847.
%
%\bibitem{LiuLiZhang2021}
%L. Liu, L. Li, W. Zhang, Iterative roots of exclusive multifunctions,
%{\it J. Difference Equ. Appl.} 27 (2021),  1, 41--60.
%
%\bibitem{LiuZhang2011}
%L. Liu,  W. Zhang, Non-monotonic iterative roots extended from characteristic intervals, {\it J.
	%	Math. Anal. Appl.}, 378 (2011), 359--373.
%
%\bibitem{Liu-zhang2021}
%L. Liu, W. Zhang, Genetics of iterative roots for PM functions, {\it Discrete Contin. Dyn. Syst.}, 41 (2021),  5, 2391--2409.
%
%\bibitem{Martin}
%R. J. Martin, M\"{o}bius splines are closed under continuous iteration, {\it Aequat.
	%	Math.}, 64 (2002),  274--296.
%
%%    \bibitem{Munkkers}
%%J. R. Munkres, {\it Topology}, Second edition,  Prentice Hall, Inc., Upper Saddle River, NJ, 2000.
%
%\bibitem{narayaninsamy2000}
%T. Narayaninsamy, A connection between fractional iteration and graph theory, {\it  Appl. Math. Comput.},  107  (2000), 2-3, 181--202.
%

%\bibitem{preston1988} 
%C. Preston, {\it Iterates of piecewise monotone mappings on an interval}, Lecture Notes in Mathematics, 1347, Springer-Verlag, Berlin, 1988.

%\bibitem{rice1980}
%R. E. Rice, B. Schweizer, A. Sklar, When is $f(f(z)) = az^2 + bz + c?$, {\it Amer. Math. Monthly},  87 (1980), 4,  252--263.
%
%%\bibitem{Shi2012}
%%Y-G. Shi, Iterative roots with circuits for piecewise continuous and
%%globally periodic maps, {\it Topology Appl.}, 159 (2012),  10-11, 2721--2727.

%\bibitem{simon1989}
%K. Simon, Some dual statements concerning Wiener measure and Baire category, {\it Proc. Amer. Math. Soc.},  106 (1989), 2, 455--463.

\bibitem{simon1989}
K. Simon, Some dual statements concerning Wiener measure and Baire category, {\it Proc. Amer. Math. Soc.},  106 (1989), 2, 455--463.

\bibitem{simon1990}
K. Simon, Typical functions are not iterates, {\it Acta Math. Hungar.},  55 (1990), 133--134.

\bibitem{simon1991a}
K. Simon, The set of second iterates is nowhere dense in $C$, {\it Proc. Amer. Math. Soc.},  111 (1991), 1141--1150.

\bibitem{simon1991b}
K. Simon, The iterates are not dense in $C$, {\it Math. Pannon.}, 2 (1991), 71--76.

%\bibitem{solarz2003}
%P. Solarz, On some iterative roots on the circle, {\it Publ. Math. Debrecen},  63 (2003), 677--692.
%
%\bibitem{solarz2006}
%P. Solarz, Iterative roots of some homeomorphisms with a rational rotation number, {\it Aequationes
	%	Math.},  72 (2006), 152--171.
%
%\bibitem{Targonski1981}
%G. Targonski, {\it Topics in Iteration Theory}, Vandenhoeck and Ruprecht, G\"{o}ttingen, 1981.
%
%\bibitem{Targonski1995}
%G. Targonski, Progress of iteration theory since 1981, {\it Aequationes Math.},  50 (1995),  50--72.
%
%\bibitem{wilf2006}
%H. S. Wilf, {\it generatingfunctionology}, Third edition, A K Peters, Ltd., Wellesley, MA, 2006.
%
%\bibitem{Xu-zhang}
%B. Xu, W. Zhang,
%Construction of continuous solutions and stability for the polynomial-like iterative equation,
%{\it J. Math. Anal. Appl.}, 325 (2007), 1160--1170.
%
%\bibitem{Yu2021}
%Z. Yu, L. Li, L. Liu, Topological classifications for a class of 2-dimensional quadratic mappings and an application to iterative roots,
%{\it Qual. Theory Dyn. Syst.}, 20 (2021), 1, Paper No. 2, 25 pp.
%
%\bibitem{zdun2000}
%M. C. Zdun, On iterative roots of homeomorphisms of the circle, {\it Bull. Polish Acad. Sci. Math.},  48 (2000), 2, 203--213.
%
%\bibitem{zdun2014}
%M. C. Zdun, On approximative embeddability of diffeomorphisms in $C^1$-flows, {\it J. Differ. Equ.
	%	Appl.},  20 (2014), 1427--1436.
%
%\bibitem{zdun-soalrz}
%M. C. Zdun, P. Solarz, Recent results on iteration theory: iteration groups and semigroups in
%the real case, {\it Aequationes Math.},  87 (2014), 201--245.
%
%\bibitem{zdun-zhang}
%M. C. Zdun, W. Zhang, Koenigs embedding flow problem
%with global $C^1$ smoothness, {\it J. Math. Anal. Appl.}, 374 (2011), 633--643.
%
%\bibitem{zhang-zeng}
%W. Zhang, Y. Zeng, W. Jarczyk, W. Zhang, Local $C^1$ stability versus global $C^1$
%unstability for iterative roots, {\it J. Math. Anal. Appl.}, 386  (2012), 75--82. 
%
%\bibitem{Zhang-zhang2007}
%W. Zhang, W. Zhang, Computing iterative roots of polygonal functions, {\it J. Comput. Appl. Math.}, 205 (2007), 1, 497--508.
%
%\bibitem{zhang-zhang}
%W. Zhang, W. Zhang,  Continuity of iteration and approximation of iterative roots, {\it J.
	%	Comput. Appl. Math.} 235 (2011), 1232--1244.
%
%\bibitem{zhang1995}
%W. Zhang, A generic property of globally smooth iterative roots, {\it Sci. China Ser.
	%	A.}, 38 (1995),  267--272.
%
%\bibitem{wzhang1997}
%W. Zhang, PM functions, their characteristic intervals and iterative roots. {\it Ann. Polon.
	%	Math.}, 65 (1997), 2, 119--128.	


%%%%%%%%%%%%%%%%%%%%%%%%%%%%%%%%%%%%%%%%	

%\bibitem{bib1}
%A. Ananin and A. Mironov,  The moduli space of $2$-dimensional algebras, \textit{Comm. Algebra}, 28 (2000), 9,  {4481}--{4488}.
%
%\bibitem{bib2}
%C. Bai and D. Meng,  The classification of Novikov algebras in low dimension,  \textit{J. Phys. A: Math. Gen}., 34 (2001), {1581}--{1594}.
%
%\bibitem{bib3}
%E. Ca\~{n}ete and A. Khudoyberdiyev,  The classification of $4$-dimensional Leibniz algebras,  \textit{Linear Algebra and its Applications}, 439 (2013), 1, {273}--{288}.
%
%\bibitem{bib4}
%M. Goze and E. Remm,  $2$-dimensional algebras,  \textit{Afr. J. Math. Phys}., 10 (2011), 1,  {81}--{91}.
%
%\bibitem{bib5}
%H. Petersson,  The classification of two-dimensional nonassociative algebras,  \textit{Results Math}., 37 (2000), no. 1-2,  {120}--{154}.
\end{thebibliography}

\end{document}